\documentclass[3p,times]{elsarticle}

\usepackage{amssymb}
\usepackage{amsmath,amsthm}
\usepackage{amssymb,latexsym}
\usepackage{enumerate}

\numberwithin{equation}{section}

\newtheorem{theorem}{Theorem}[section]

\newtheorem{proposition}[theorem]{Proposition}

\newtheorem{corollary}[theorem]{Corollary}
\newtheorem{The main theorem}[theorem]{The main theorem}

\theoremstyle{definition}

 \def\f{\mathcal F\text{-}}

\allowdisplaybreaks

\begin{document}

\begin{frontmatter}



%% e.g. \dochead{17th International Conference on Dynamical Processes in Excited States of Solids}

\title{Local property of  maximal plurifinely plurisubharmonic functions}


\author[label1]{Nguyen Xuan Hong \fnref{label3}} 
 \fntext[label3]{This work is finished during the first author's post-doctoral fellowship of the Vietnam Institute for Advanced Study in Mathematics. He wishes to thank this institution for their kind hospitality and support.}
\address[label1]{Department of Mathematics, Hanoi National University of Education, 136 Xuan Thuy Street, Caugiay District, Hanoi, Vietnam}
 \ead{xuanhongdhsp@yahoo.com}
 
\author[label2]{Hoang Viet}\address[label2]{Vietnam Education Publishing House, Hanoi, Vietnam}  \ead{viet.veph@gmail.com}

\begin{abstract}
In this paper, we prove that a continuous $\f$plurisubharmonic functions defined in an $\f$open set in $\mathbb C^n$ is $\f$maximal if and only if it is $\f$locally $\f$maximal.  
\end{abstract}


\end{frontmatter}

% \linenumbers


\section{Introduction}

The plurifine topology $\mathcal F$ on a Euclidean open set $\Omega$ of $\mathbb C^n$ is the smallest topology that makes all plurisubharmonic function on $\Omega$ continuous.   El Kadiri \cite{K03}  defined  in 2003 the notion of finely plurisubharmonic function in a plurifine open subset of $\mathbb C^n$ and studied properties of such functions.
These functions are introduced as plurifinely upper semicontinuous functions, of which the restriction to complex lines are finely subharmonic, where
a finely subharmonic function is defined on on a fine domain is a finely upper semi-continuous and satisfies an appropriate modification of the mean value inequality. In  \cite{MW10} El Marzguioui and Wiegerinck studied the continuity properties of the plurifinely plurisubharmonic functions. 
%El  Marzguioui and   Wiegerinck \cite{MW10} introduced  in 2010  the  plurifinely plurisubharmonic functions as plurifinely upper semicontinuous functions, of which the restriction to complex lines is finely subharmonic, where a finely subharmonic function is defined on a fine domain, it is finely upper semi-continuous, and satisfies an appropriate modification of the mean value inequality.
El Kadiri,  Fuglede and Wiegerinck \cite{KFW11} proved in 2011 the most important properties of the  plurifinely plurisubharmonic functions. 
El Kadiri and Wiegerinck  \cite{KW14}  defined in 2014 the Monge-Amp\`{e}re operator on finite   plurifinely plurisubharmonic functions in plurifinely open sets and show that it defines a positive measure. 
El Kadiri and M. Smit \cite{KS14} 
introduced and  studied in 2014 the notion of $\mathcal F$-maximal $\mathcal F$-plurisubharmonic functions, which extends the notion of maximal plurisubharmonic functions on a Euclidean domain to an $\mathcal F$-domain of $\mathbb C^n$ in a natural way.

There is  a natural questions that  whether  an $\f$locally $\f$maximal $\f$plurisubharmonic function on an $\f$open set $\Omega$ of $\mathbb C^n$ also $\f$maximal in $\Omega$ (see question 4.17 in \cite{KS14}).  
El Kadiri and M. Smit \cite{KS14} gave an example to show that this result is not valid when the function is not finite.
The aim of  this paper is to give a positive  answer for this question when the function is continuous.
Namely, we will prove the following theorem.\\

\noindent 
{\bf  Main theorem.} {\em 
Let $\Omega$ be an $\f$open set in $\mathbb C^n$. Assume that    $u$ is a continuous $\f$plurisubharmonic function in $\Omega$.  
Then $u$ is $\f$maximal  in $\Omega$ if and only if it is $\f$locally $\f$maximal   in $\Omega$.
} \\

Klimek \cite{Kl91} proved that  a locally bounded plurisubharmonic function $u$ defined in an Euclidean open set is maximal  if and only if  $(dd^cu)^n =0$, and therefore, the bounded  plurisubharmonic functions defined  in an Euclidean open set is maximal if and only if it is locally maximal. Notice that  for bounded $\f$plurisubharmonic functions $u$ defined in an $\f$open set $\Omega$, the complex Monge-Amp\`ere operator $(dd^c u)^n$ is $\f$locally defined in $\Omega$ (see \cite{KW14}), and therefore,  $(dd^cu)^n =0$ in $\Omega$ if and only if $u$ is $\f$locally $\f$maximal in $\Omega$ (see \cite{KS14}). Hence, it need to find an another approach in studying the local property of $\f$maximal $\f$plurisubharmonic functions.
Techniques used in the proof of the main theorem come from  \cite{HHo131} (also see \cite{HHo132}).

The paper is organized as follows. In section 2 we recall some notions of plurifine pluripotential theory. In Section 3 we prove  main theorem.

\section{Preliminaries}
Some elements of pluripotential theory (plurifine pluripotential theory) that will be used throughout the paper can be found in \cite{Bl2009}-\cite{W12}.\\

\noindent {\bf 2.1. The plurifine topology} 

The plurifine topology $\mathcal F$ on a
Euclidean open set $\Omega$ of $\mathbb C^n$  is the smallest topology that makes all plurisubharmonic functions on $\Omega$ continuous. 

Notions pertaining to the plurifine topology are indicated with the prefix $\mathcal F$ to distinguish them from notions pertaining to the Euclidean topology on $\mathbb C^n$.
For a set $A\subset \mathbb  C^n$ we write $\overline{A}$ for the closure of $A$ in the one point compactification
of $\mathbb C^n$, $\overline{A}^{\mathcal F}$ for the $\mathcal F$-closure of $A$ and $\partial _{\mathcal F}A$ for the $\mathcal F$-boundary of $A$.

A local  basis   is given by the sets 
$\mathbb B(z,r_z)\cap \{\varphi_z>0\},$
where $\mathbb B(z,r_z) \subset\Omega$ be Euclidean open balls of center  $z$, radius $r_z$; and $\varphi_z \in  PSH(\mathbb B(z,r_z))$ with $\varphi_z(z) > 0$. 

 The plurifine topology  is quasi-Lindel{\"o}f, that is, every arbitrary union of $\f$open sets is
the union of a countable subunion and a pluripolar set.\\
 
\noindent {\bf 2.2.   $\f$plurisubharmonic functions}

Let $\Omega$ be an $\f$open set in $\mathbb C^n$. A function $u : \Omega \to [-\infty,+\infty)$ is said to be $\f$plurisubharmonic function if it is $\f$upper semicontinuous, and for every complete line $l$ in $\mathbb C^n$, the restriction of $u$ to any $\f$component of the finely open subset $l\cap \Omega$ of $l$ is either finely subharmonic or $\equiv -\infty$. 
 
The set of all  $\f$plurisubharmonic functions in $\Omega$ is denoted by $\f  PSH(\Omega)$.

Let  $u\in \f PSH(\Omega)$.  We say that $u$  is $\f$maximal  in $\Omega$  if for every bounded $\f$open set $G$ of $\mathbb C^n$ with $\overline G\subset \Omega$, and for every function  $v\in \f PSH(G)$ that is bounded from above on $G$ and extends $\f$upper semicontinuously to $\overline{G}^{\mathcal F}$ with   $v\leq u$ on  $\partial_{\mathcal F} G$ implies   $v\leq u$ on $G$. 

The set of all $\f$maximal $\f$plurisubharmonic functions in $\Omega$ is denoted by $\f MPSH(\Omega)$. 

The function $u$ is called   locally  (resp. $\f$locally) $\f$maximal in $\Omega$ if for every $z\in \mathbb C^n$ there exists an Euclidean open (resp. $\f$open) neighbourhood $V_z\subset \mathbb C^n$ of $z$ such that $u|_{V_z\cap \Omega}$ is $\f$maximal   on $V_z\cap \Omega$.

\section{Proof of main theorem}

First we give the following.

\begin{proposition}\label{pr1}
Let $\Omega$ be an $\f$open set in $\mathbb C^n$. Assume that    $u$ is a bounded $\f$plurisubharmonic function in $\Omega$.   
Then, the following conditions are equivalent 

(a) $u\in \f MPSH(\Omega)$.

(b) $u+g\in \f MPSH(\Omega)$, for every pluriharmonic functions $g$ in $\mathbb C^n$.

(c) For every $v\in \f PSH(\Omega)$ and for every $\f$open set $G\subset \Omega$ with $\overline G\subset \Omega$ we have 
$$\sup_{G} (v-u) 
\leq \sup_{\Omega \backslash G} (v-u).$$  
\end{proposition}

\begin{proof} 
(a) $\Leftrightarrow$ (b) is obvious.

(a) $\Rightarrow$ (c). Let $v\in \f PSH(\Omega)$ and let $G$ be an $\f$open set with $\overline G\subset \Omega$. Put 
$$M:= \sup_{\Omega \backslash G} (v-u).$$
Without loss of generality we can assume that $M<+\infty$. Then $v-M \leq u$ in $\Omega \backslash G$. In particular,   $v-M\leq u$ on  $\partial_{\mathcal F} G$, and hence, 
$v-M \leq u$ in $G$. Therefore,
$$\sup_{G} (v-u) 
\leq M= \sup_{\Omega \backslash G} (v-u).$$ 

(c) $\Rightarrow$ (a).
Let $G$ be an $\f$open set in  $\mathbb C^n$ with $\overline G\subset \Omega$, and let  $v\in \f PSH(G)$ such that $v$ is  bounded from above  on $G$, extends $\f$upper semicontinuously to $\overline{G}^{\mathcal F}$, and    $v\leq u$ on  $\partial_{\mathcal F} G$.
Put 
$$\varphi := \begin{cases} 
\max(v,u) & \text{ on } G,
\\ u & \text{ on } \Omega \backslash G.
\end{cases}$$
Thanks to Proposition 2.3 in \cite{KS14} we have  $\varphi \in \f PSH(\Omega)$. It follows that 
$$\sup_{G} (v-u) 
\leq \sup_{G} (\varphi-u) 
\leq  \sup_{\Omega \backslash G} (\varphi-u)=0.$$
Hence, $v\leq u$ in $G$. The proof is complete.  
\end{proof}

\begin{proposition}\label{pr3}
Let $\Omega$ be  $\f$open sets in $\mathbb C^n$. Assume that     $u$ is   bounded,   locally $\f$maximal,  $\f$plurisubharmonic function in $\Omega$. Then  $u$ is $\f$maximal   in $\Omega$.
\end{proposition}

\begin{proof} 
Let $v$ be a $\f$plurisubharmonic function in $\Omega$ and let $G$ is a bounded $\f$open set  in $\mathbb C^n$ such that  $\overline G\subset \Omega$.  
Choose $R>0$ such that $\overline G\subset \mathbb B(0,R)$.
Let $\varepsilon>0$. Put $v_\varepsilon(z):=v(z)+\varepsilon |z|^2$, $z\in \Omega$. 
Choose $\{p_j \}\subset   G$ such that  $p_j\to p \in \overline{G}$ and 
$$\lim_{j\to+\infty} [v_\varepsilon(p_j)- u(p_j)] =\sup_{G} (v_\varepsilon-u).$$ 
Let $r>0$ such that $\mathbb B(p,3r)\Subset \mathbb B(0,R)$ and  $u$ is $\f$maximal $\f$plurisubharmonic function in  $\mathbb B(p,3r)\cap \Omega$. Without loss of generality we can assume that $\{p_j\} \subset \mathbb B(p,r)$. Put
$$g_{\varepsilon,j} (z) := \varepsilon |z-p_j|^2 -\varepsilon |z|^2, \ z\in \mathbb C^n.$$
It is clear that $g_{\varepsilon,j}$ are pluriharmonic functions in $\mathbb C^n$. 
Following  Proposition \ref{pr1} we have  $u+g_{\varepsilon,j}$ are $\f$maximal $\f$plurisubharmonic functions in $\mathbb B(p,3r)\cap \Omega$, and hence, again by Proposition \ref{pr1}, we get
\begin{align*}
\sup_{G} (v_\varepsilon-u) 
& =   \lim_{j\to+\infty} [v(p_j) - u(p_j) -g_{\varepsilon,j}(p_j) ]  
\\ & \leq  \sup_{G\cap \mathbb B(p,2r)} (v- u -g_{\varepsilon,j}) 
\\&  \leq  \sup_{(\Omega  \cap \mathbb B(p,3r) ) \backslash (G\cap \mathbb B(p,2r))} (v- u -g_{\varepsilon,j}) 
\\& \leq 
\max
\left(
\sup_{(\Omega \cap \mathbb B(p,3r)) \backslash G} (v- u -g_{\varepsilon,j}) ,
\sup_{G\backslash \mathbb B(p,2r)} (v- u -g_{\varepsilon,j}) 
\right) 
\\& \leq 
\max
\left(
\sup_{(\Omega \cap \mathbb B(p,3r))  \backslash G} (v_\varepsilon- u) ,
\sup_{G\backslash \mathbb B(p,2r)} (v_\varepsilon- u -\varepsilon r^2) 
\right) .
\end{align*}
Therefore, 
\begin{align*}
\sup_{\Omega \cap \mathbb B(0,R)} (v_\varepsilon- u) 
&= \max( \sup_{G \cap \mathbb B(0,R)} (v_\varepsilon-u) , \sup_{(\Omega \cap \mathbb B(0,R)) \backslash G} (v_\varepsilon- u) )
\\& \leq   \max( \sup_{\Omega \cap \mathbb B(p,3r)} (v_\varepsilon- u) -\varepsilon r^2 , \sup_{(\Omega \cap \mathbb B(0,R)) \backslash G} (v_\varepsilon- u) )
\\& \leq   \max(\sup_{\Omega \cap \mathbb B(0,R)} (v_\varepsilon- u)  -\varepsilon r^2 , \sup_{(\Omega \cap \mathbb B(0,R)) \backslash G} (v_\varepsilon- u) ).
\end{align*}
It follows that
\begin{align*}
\sup_{\Omega \cap \mathbb B(0,R)} (v_\varepsilon- u)  =   \sup_{(\Omega \cap \mathbb B(0,R)) \backslash G} (v_\varepsilon- u)  .
\end{align*}
Hence,
\begin{align*}
\sup_{G} (v-u)&  \leq \sup_{G} (v_\varepsilon-u) 
 \leq  \sup_{\Omega \cap \mathbb B(0,R)} (v_\varepsilon- u) -\varepsilon r^2 
\\&= \sup_{( \Omega \cap \mathbb B(0,R))  \backslash G} (v_\varepsilon-u) -\varepsilon r^2 
\\& <\sup_{( \Omega \cap \mathbb B(0,R))  \backslash G} (v-u) +\varepsilon R^2  
 \leq \sup_{\Omega\backslash G} (v -u) + \varepsilon R^2.
\end{align*}
Let $\varepsilon \searrow 0$ we obtain that 
$$\sup_{G} (v-u) 
\leq \sup_{\Omega \backslash G} (v-u).$$
Thanks to  Proposition \ref{pr1} this  implies that $u$ is $\f$maximal $\f$plurisubharmonic function in $\Omega$. The proof is complete.
\end{proof}

%\begin{corollary}\label{Coro3.3}
%Let $\Omega$ be an $\f$open set in $\mathbb C^n$ and let $u$ is a finite  $\f$plurisubharmonic function in $\Omega$.  
%Then $u$ is $\f$locally $\f$maximal  in $\Omega$ if and only if there exists an increasing sequence $\{V_j\}$ of $\f$open sets in $\Omega$ and a pluripolar set $F$ in $\Omega$ such that 
%$\Omega=F\cup \bigcup_{j=1}^\infty V_j$ and  $u$ is $\f$maximal on each $V_j$. 
%\end{corollary} 

%\begin{proof}

%Assume that u is $\f$locally $\f$maximal in $\Omega$. Then there exists a $\f$closed pluripolar set $F_1\subset \mathbb C^n$ such that $u$ is finite on $\Omega \backslash F_1$. For any $j\in\mathbb N$, $j\neq 0$, write 
%$$V_j =\{-j<u<j\}.$$ 
%Then $V_j$ is a $\f$open set and, according to Proposition \ref{pr3}, $u$ is $\f$maximal in $V_j$ (because $u$ is bounded on $V_j$). The sequence $\{V_j\}$ is clearly increasing and we have $\Omega = F \cup \bigcup_{j} V_j$, where $F = F_1 \cap \Omega$.

%Conversely, by Theorem 4.8 in \cite{KS14} implies that $(dd^c u)^n =0$ on $V_j$, and hence, $(dd^c u)^n=0$ on $\Omega$. Therefore, by Theorem 4.15 in \cite{KS14} we obtain that $u$ is $\f$locally $\f$maximal in $\Omega$. The proof is complete.
%\end{proof}

We now give the proof of main theorem.

\begin{proof}[Proof of main theorem]
The proof of the necessity is obvious. We now give the proof of the sufficiency. By Proposition \ref{pr3} it remains to prove that $u$ is locally $\f$maximal  in $\Omega$, and hence, without loss of generality  we can assume that $\Omega\subset \mathbb B(0,R)$. 
Let $G$ is a bounded $\f$open set  in $\mathbb C^n$ with $\overline G\subset \Omega$, and let  $v\in \f PSH(G)$ such that $v$ is  bounded from above  on $G$, extends $\f$upper semicontinuously to $\overline{G}^{\mathcal F}$ and    $v\leq u$ on  $\partial_{\mathcal F} G$. 
Let $\varepsilon>0$. 
Put 
$$v_\varepsilon(z) := \begin{cases} 
\max(v(z) +\varepsilon(|z|^2 -R^2) ,u(z)) & \text{ if  } z\in G,
\\ u(z)  & \text{ if }z\in \Omega \backslash G.
\end{cases}$$
Thanks to Proposition 2.3 in \cite{KS14} we have  $v_\varepsilon \in \f PSH(\Omega)$. 
Assume that 
$$\sup_{G} (v_\varepsilon-u) > \delta_0>0.$$
Choose $\{p_j \}\subset   G$ such that  $p_j\to p \in \overline{G}$, $v_\varepsilon(p_j)- u(p_j) > \delta_0$ for all $j\geq 1$
and 
$$\lim_{j\to+\infty} [v_\varepsilon(p_j)- u(p_j)] 
=\sup_{G} (v_\varepsilon-u)
=\sup_{\Omega} (v_\varepsilon-u).$$ 

First, we claim that 
$$\limsup_{j\to+\infty} v_\varepsilon (p_j)\leq v_\varepsilon (p).$$ 
Indeed, let $\delta\in (0,\delta_0)$.
Since $u$ is continuous in $\overline G$ so there exist a smooth functions $f$ defined in $\mathbb C^n$ such that 
$$u \leq f  \leq u  +\delta \text{ on }  \overline{G}.$$
Choose $\varphi, \psi \in PSH(\mathbb B(0,R)) \cap \mathcal C^\infty (\mathbb B(0,R))$ such that 
$$f=\varphi -\psi \text{ on } \mathbb B(0,R).$$
Put 
$$w := \begin{cases} 
\max(v_\varepsilon +\psi ,f+\psi) & \text{ in  } G,
\\ f +\psi  & \text{ in } \mathbb B(0,R) \backslash G.
\end{cases}$$
Since $v_\varepsilon \leq u\leq f$ on $ \partial _{\mathcal F} G$, from Proposition 2.3 in \cite{KS14} we have  $w\in \f PSH(\mathbb B(0,R))$, and hence, by Proposition 2.14 in \cite{KFW11} we have 
$w\in PSH(\mathbb B(0,R))$.
Therefore, $w$ be upper semicontinuous function in $\mathbb B(0,R)$. Since 
$$v_\varepsilon(p_j) -f(p_j) \geq v_\varepsilon(p_j) -u(p_j) -\delta >0,
$$
we have $w(p_j)=v_\varepsilon(p_j) +\psi(p_j)$. It follows that 
\begin{align*}
\limsup_{j\to+\infty} v_\varepsilon (p_j) 
&= \limsup_{j\to +\infty} [w(p_j) -\psi(p_j)] 
\\& \leq w(p)-\psi(p)
\leq \max(v_\varepsilon(p), f(p))
\\& \leq \max(v_\varepsilon(p), u(p) +\delta).
\end{align*}
Letting  $\delta \searrow 0$ we obtain that 
$$\limsup_{j\to+\infty} v_\varepsilon (p_j)
\leq  \max(v_\varepsilon(p), u(p) )= v_\varepsilon (p).$$
This proves the claim.

Now, since $u$ is continuous function, we  have 
\begin{align*}
v_\varepsilon(p)-u(p) 
&\leq \sup _\Omega (v_\varepsilon -u)
=\limsup_{j\to+\infty} [v_\varepsilon(p_j) -u(p_j)] 
\\& = \limsup_{j\to+\infty} v_\varepsilon (p_j) -u(p) 
\\&\leq v_\varepsilon(p)-u(p) .
\end{align*}
Therefore, 
$$v_\varepsilon(p)- u(p) 
=\sup_{G} (v_\varepsilon-u)
=\sup_{\Omega} (v_\varepsilon-u)>0.$$ 
It follows that $p \in \Omega \cap \{v_\varepsilon > u\}$.
Thanks to Theorem 3.1 in \cite{MW10} this  implies that $\Omega\cap \{v_\varepsilon>u\}$ is $\f$open neighbourhood of $p$.

Let $r>0$ and let $\phi \in PSH(\mathbb B(p,2r))$ such that $\phi(p)=0$, $\mathbb B(p,2r)\cap \{\phi>-2\} \subset \Omega\cap \{v_\varepsilon>u\}$ and  $u$ is $\f$maximal $\f$plurisubharmonic function in  $\mathbb B(p,2r)\cap \{\phi>-2\}$. 
Without loss of generality we can assume that $\phi$ is bounded on $\mathbb B(p,2r)$.
Put
$$g (z) := \varepsilon |z-p|^2 -\varepsilon (|z|^2-R^2), \ z\in \mathbb C^n.$$
Since $g$ is a pluriharmonic functions in $\mathbb C^n$, by Proposition \ref{pr1} we have  $u+g$ is $\f$maximal $\f$plurisubharmonic function in $\mathbb B(p,2r)\cap \{\phi>-2\}$. Let $\delta>0$. Again by Proposition  \ref{pr1}, we get
\begin{align*}
 \sup_{\Omega} (v_\varepsilon-u)  
 &= v_\varepsilon(p)-u(p) 
\\ & =   v(p)+ \delta \phi (p) - u(p) -g(p)  
\\&  \leq  \sup_{(\mathbb B(p,2r)\cap \{\phi>-2\}) \backslash (\mathbb B(p,r)\cap \{\phi>-1\})} (v + \delta \phi - u -g) 
\\ &\leq \max
\left( 
 \sup_{\mathbb B(p,2r)\cap \{-2 < \phi \leq -1\}} (v + \delta \phi - u -g) ,
 \sup_{(\mathbb B(p,2r)\cap \{\phi>-2\}) \backslash \mathbb B(p,r)} (v + \delta \phi - u -g) 
\right) 
\\ &\leq \max
\left( 
 \sup_{\mathbb B(p,2r)\cap \{-2 < \phi \leq -1\}} (v _\varepsilon - \delta   - u ) ,
 \sup_{(\mathbb B(p,2r)\cap \{\phi>-2\}) \backslash \mathbb B(p,r)} (v _\varepsilon + \delta \phi - u -\varepsilon r^2) 
\right)  
\\ &\leq \max
\left( 
 \sup_{\Omega} (v _\varepsilon    - u ) - \delta ,
 \sup_{\mathbb B(p,2r)\cap \{\phi>-2\}} (v _\varepsilon + \delta \phi - u -\varepsilon r^2) 
\right)  
\end{align*}
It implies that 
\begin{align*}
 \sup_{\Omega} (v_\varepsilon-u)   
& \leq  
 \sup_{\mathbb B(p,2r)\cap \{\phi>-2\}} (v _\varepsilon + \delta \phi - u -\varepsilon r^2) 
 \\&  \leq  \sup_{\Omega} (v_\varepsilon-u)   + 
 \delta \sup_{\mathbb B(p,2r)}  \phi  -\varepsilon r^2.
\end{align*}
Let $\delta \searrow 0$, we obtain 
\begin{align*}
 \sup_{\Omega} (v_\varepsilon-u)   
   \leq  \sup_{\Omega} (v_\varepsilon-u)    -\varepsilon r^2 .
\end{align*}
This is impossible. Thus,  
$$\sup_{G} (v_\varepsilon-u)   \leq 0.$$
Therefore, 
$$\sup_{G} (v-u) 
\leq \sup_{G} (v_\varepsilon-u)  +\varepsilon R^2 \leq  \varepsilon R^2.$$
Letting $\varepsilon \searrow 0$ we obtain that 
$$\sup_{G} (v-u) 
\leq 0.$$
It follows  that $v\leq u$ in $G$. Hence, $u$ is $\f$maximal $\f$plurisubharmonic function in $\Omega$. The proof is complete.
\end{proof}

Main Theorem  together with Proposition 2.5 in \cite{KS14} and Theorem 4.15 in \cite{KS14} gives

\begin{corollary}
Let $\Omega$ be an $\f$open set in $\mathbb C^n$. Assume that   $u$ is a continuous $\f$plurisubharmonic function in $\Omega$.  
Then $u$ is $\f$maximal  in $\Omega$ if and only if $(dd^c u)^n=0$ on $\Omega$.
\end{corollary}



%% References


\bibliographystyle{elsarticle-num}
\bibliography{<your-bib-database>}



\end{document}